\documentclass[11pt]{amsart}
\usepackage{amssymb,pb-diagram}
\usepackage{amsmath}
\usepackage{amsthm,amssymb,amsfonts}
\usepackage{amssymb,amscd}
\usepackage{epic,eepic,epsfig}
\usepackage{latexsym}
\usepackage{stmaryrd}
\usepackage{xy}
\xyoption{all}

\usepackage{amssymb,amsthm,enumerate,color,mathrsfs}

\newtheorem{Theorem}{Theorem}[section]

\newtheorem{Proposition}[Theorem]{Proposition}
\newtheorem{Lemma}[Theorem]{Lemma}
\newtheorem{Corollary}[Theorem]{Corollary}

\theoremstyle{definition}
\newtheorem{Definition}[Theorem]{Definition}

\newtheorem{Notation}[Theorem]{Notation}

\theoremstyle{definition}

\theoremstyle{remark}

\def\supp{\text{supp}}

\newcommand{\bib}{\bibitem}

\def\supp{\text{supp }}

\begin{document}

\title{Coarse Lipschitz embeddings of James spaces}

\author{F. Netillard}
\address{Universit\'{e} de Franche-Comt\'{e}, Laboratoire de Math\'{e}matiques UMR 6623,
16 route de Gray, 25030 Besan\c{c}on Cedex, FRANCE.}
\email{francois.netillard@univ-fcomte.fr}




\maketitle

\begin{abstract} We prove that, for $1 < p \neq q < \infty$, there does not exist any coarse Lipschitz embedding between the two James spaces $J_p$ and $J_q$, and that, for $1 < p < q < \infty$ and $1 < r < \infty$ such that $r \notin \{p,q\}$, $J_r$ does not coarse Lipschitz embed into $J_p \oplus J_q$.
\end{abstract}

\markboth{}{}

\section{Introduction}

\noindent Let $(M,d)$ and $(N,\delta)$ be two metric spaces and $f~:M \to N$.\\
The map $f$ is said to be a \textit{coarse Lipschitz embedding} if there exist $\theta$, $A$, $B > 0$ such that $$\forall~ x,~y \in M\ \ d(x,y) \geq \theta \Rightarrow Ad(x,y) \leq \delta(f(x),f(y)) \leq Bd(x,y).$$

\noindent Then we say that $M$ \textit{coarse Lipschitz embeds} into $N$.

\medskip\noindent R.C. James introduced in \cite{J} a non-reflexive space defined by~: $$J = \left\{x: \mathbb{N} \rightarrow \mathbb{R}~ \textnormal{s.t.} ~x(n) \rightarrow 0~ \textnormal{and}~ ||x||_{J} = \underset{p_1 < ... < p_n}{sup}\left(\sum\limits_{i=1}^{n-1}|x(p_{i+1}) - x(p_i)|^2\right)^{\frac{1}{2}} < \infty\right\}$$

\noindent We will use the following spaces (where $1 < p < \infty$), which are variants of $J$~: $$J_p = \left\{x: \mathbb{N} \rightarrow \mathbb{R}~ \textnormal{s.t.} ~x(n) \rightarrow 0~ \textnormal{and}~ ||x||_{J_p} = \underset{p_1 < ... < p_n}{sup}\left(\sum\limits_{i=1}^{n-1}|x(p_{i+1}) - x(p_i)|^p\right)^{\frac{1}{p}} < \infty\right\}$$

\noindent Like in the case of $J$, the codimension of $J_p$ in $J_p^{**}$ is $1$.

\noindent  In this respect,  we precise that $J_p^{\ast\ast}$ can be seen as~: $$J_p^{\ast\ast} = \left\{x: \mathbb{N} \rightarrow \mathbb{R}~ \textnormal{s.t.} ~ \underset{p_1 < ... < p_n}{sup}\left(\sum\limits_{i=1}^{n-1}|x(p_{i+1}) - x(p_i)|^p\right)^{\frac{1}{p}} < \infty\right\}$$

\noindent All those spaces were introduced in \cite{P}.

\medskip\noindent In $2007$, N.J. Kalton and N.L. Randrianarivony [5] proved that, if $r \notin \{p_1,~...,~p_n\}$ where $1 \leq p_1 < p_2 < ... <p_n < \infty$, then $\ell_r$ does not coarse Lipschitz embed into $\ell_{p_1} \oplus ... \oplus \ell_{p_n}$.

\noindent The aim of this article is to prove similar results for the $J_p$ spaces. One of the main obstacles is the lack of reflexivity, which was crucial in Kalton-Randrianarivony's work. However, the James spaces have nice properties of asymptotic uniform smoothness and weak$^*$ asymptotic uniform convexity that we shall use (see \cite{KR} or \cite{L1} for the definitions). We shall not refer to these notions in our paper, but we will build concrete equivalent norms on $J_p$ that will serve our purpose. Some compactness arguments will also be used to deal with the extra dimension in $J_p^{**}$.

\medskip\noindent This paper is organized as follows. In Section 2 we summarize the notation and terminology and we give the basic results. Section 3 contains the proof of the nonexistence of a coarse Lipschitz embedding between two James spaces $J_p$ and $J_q$ for $1 < p \neq q < \infty$. At the end of this last section,  we show that, for $1 < p < q < \infty$ and $1 < r < \infty$ such that $r \notin \{p,q\}$, $J_r$ does not coarse Lipschitz embed into $J_p \oplus J_q$.

\section{Preliminaries}

\begin{Notation}
Let $e_n$ defined by $e_n(k) = \delta_{n,k}$ for $k \in \mathbb{N}$. The sequence $(e_n)_{n = 0}^{\infty}$ is a Schauder basis of $J_p$ (where $p > 1$).

\noindent Moreover, the sequence $(e_n^{*})_{n=0}^{\infty}$ of the coordinate functionals associated with $(e_n)_{n=0}^{\infty}$ is a Schauder basis of $J_p^{*}$.

\noindent When $u$ and $v$ in $J_p$ have a consecutive and disjoint finite supports with respect to $(e_n)_{n=0}^{\infty}$, we will denote $u \prec v$.

\noindent Likewise, when $u^{*}$ and $v^{*}$ in $J_p^{*}$ have a consecutive and disjoint finite supports with respect to $(e_n^{*})_{n=0}^{\infty}$, we will denote $u^{*} \prec v^{*}$.
\end{Notation}

We start with the construction of an ad'hoc equivalent norm on $J_p$. We follow the construction given in \cite{L2} for $J_2$.

\begin{Lemma}\label{L1}
Let $x_1,~...,~x_n$ in $J_p$ such that their supports are consecutive and finite with respect to the basis $(e_n)_{n=0}^{\infty}$. Then
$$||\sum\limits_{i=1}^{n}x_i||_{J_p}^p \leq (2^p+1)\sum\limits_{i=1}^{n}||x_i||_{J_p}^p.$$
\end{Lemma}

\begin{proof}
For $x \in J_p$, we denote $supp(x) = \{n \in \mathbb{N},~ e_n^{*}(x) \neq 0\}$. Then we can find disjoint intervals in $\mathbb{N},~ \llbracket p_i,~ p'_i \rrbracket$, with $1 \leq i \leq n$ and $p'_i < p_{i+1}$, such that~:

\noindent $\forall~ 1\leq i \leq n,~ supp(x_i) \subset \llbracket p_i,~ p'_i \rrbracket$ (for convenience, we fix $p_1 = 0$ et we denote $p_{n+1} = \infty$).

\noindent Let now $q_1 < ... < q_k$ be an arbitrary sequence in $\mathbb{N}$. We must show that $$\sum\limits_{j=1}^{k-1}\left|y(q_j) - y(q_{j+1})\right|^p \leq (2^p+1)\sum_{i=1}^{n}||x_i||_{J_p}^p,$$ where $y = \sum\limits_{i=1}^{n}x_i$.

\noindent There exist an increasing sequence $(j_m)_{m=1}^{l}$ in $\{1,~...,~k\}$ with $j_1 = 1$, and an increasing sequence $(i_m)_{m=1}^{l}$ in $\{1,~...,~n\}$ such that, for any $1 \leq m \leq l$, $\{q_{j_m},~...,~q_{j_{m+1}}-1\} \subset~\llbracket~p_{i_m},~ p_{i_{m+1}}\llbracket$). Therefore
$$\sum\limits_{j=1}^{k-1}\left|y(q_j) - y(q_{j+1})\right|^p=$$
$$\sum\limits_{j=1}^{j_2 - 1}\left|y(q_j) - y(q_{j+1})\right|^p + \left|y(q_{j_2 - 1}) - y(q_{j_2})\right|^p + ... + \sum\limits_{j=j_{l-1}}^{j_l-1}\left|y(q_j) - y(q_{j+1})\right|^p.$$
We have that
$$\forall~ 1 \leq m \leq l-1, \sum\limits_{j=j_m}^{j_{m+1}-1}\left|y(q_j) - y(q_{j+1})\right|^p \leq ||x_{i_m}||_{J_p}^p.$$

\noindent And for all  $2 \leq m \leq l-1$,
$$\left|y(q_{j_m-1}) - y(q_{j_m})\right|^p \leq 2^{p-1}\left|y(q_{j_m - 1})\right|^p + 2^{p-1}\left|y(q_{j_m})\right|^p \leq 2^{p-1}||x_{i_m}||_{J_p}^p + 2^{p-1}||x_{i_{m+1}}||_{J_p}^p.$$
Then
$$\sum\limits_{j=1}^{k-1}(y(q_j) - y(q_{j+1}))^p\leq (2^{p-1}+1)||x_{i_1}||_{J_p}^p + (2^p+1)||x_{i_2}||_{J_p}^p + ... + (2^p+1)||x_{i_{l-1}}||_{J_p}^p + (2^{p-1}+1)||x_{i_l}||_{J_p}^p.$$
This concludes the proof of our Lemma.
\end{proof}

\noindent We now define a new norm on $J_p^*$ as follows. Let $q$ be the conjugate exponent of $p$. For $x^{*} \in J_p^{*}$, we set
$$|x^{*}|_{J_p^{*}} = \sup\Big\{(\sum\limits_{i=1}^{n}||x_i^{*}||^q_{J_p^{*}})^{\frac{1}{q}}~: x^{*} = x_1^{*} + ... + x_n^{*} ~et~ x_1^{*} \prec ... \prec x_n^{*} \Big\},$$
where $||.||_{J_p^{*}}$ denotes the dual norm of $||.||_{J_p}$. We can now state the following proposition.

\begin{Proposition}\label{P0} The norm $|.|_{J_p^{*}}$ is the dual norm of an equivalent norm on $J_p$ (that we shall denote $|.|_{J_p}$).\\
Moreover, $|.|_{J_p^{*}}$ satisfies the following property: for any $x^*,y^*$ in $J_p^*$ such that $x^{*} \prec y^{*}$, we have that
$$|x^{*} + y^{*}|^q_{J_p^{*}} \geq |x^{*}|^q_{J_p^{*}} + |y^{*}|^q_{J_p^{*}}.$$
\end{Proposition}

\begin{proof} To show that $|.|_{J_p^*}$ is a norm, we only detail the proof of the triangle inequality~: let $(x^*,~y^*) \in \left(J_p^*\right)^2$ that we may assume with finite supports. Let now $u_1^* \prec u_2^* \prec ... \prec u_n^*$ in $J_p^*$ such that
$$x^* + y^* = u_1^* + ... + u_n^*$$.

\noindent We write, for $i \in \llbracket 1,~ n\rrbracket$, $u_i^* = x_i^* + y_i^*$, where $x^* = \sum\limits_{i=1}^{n}x_i^*$ and $y^* = \sum\limits_{i=1}^{n}y_i^*$.

\noindent Thanks to the triangle inequality for $||.||_{J_p^*}$, we get:

$$\left(\sum\limits_{i=1}^{n}||u_i^*||_{J_p^*}^q\right)^{\frac{1}{q}} \leq \left(\sum\limits_{i=1}^{n}\left(||x_i^*||_{J_p^*} + ||y_i^*||_{J_p^*}\right)^q\right)^{\frac{1}{q}}.$$

\noindent It then follows from Minkowski's inequality that

$$\left(\sum\limits_{i=1}^{n}||u_i^*||_{J_p^*}^q\right)^{\frac{1}{q}} \leq \left(\sum\limits_{i=1}^{n}||x_i^*||_{J_p^*}^q\right)^{\frac{1}{q}} + \left(\sum\limits_{i=1}^{n}||y_i^*||_{J_p^*}^q\right)^{\frac{1}{q}} \leq |x^*|_{J_p^*} + |y^*|_{J_p^*}.$$

\noindent We have shown that the triangle inequality is valid for $|.|_{J_p^*}$.

\medskip Next we show that there exists $c > 0$ (which will be detailed later) such that, for $x_1^*$, ..., $x_n^*$ in $J^{*}$ satisfying $x_1^* \prec ... \prec x_n^*$ with respect to the basis $(e_n^*)_{n=0}^{\infty}$~:
\begin{equation}\label{1}
||\sum\limits_{i=1}^{n}x_i^{*}||_{J_p^{*}}^q \geq c\sum\limits_{i=1}^{n}||x_i^{*}||_{J_p^{*}}^q.
\end{equation}

\noindent So, let $x_1^* \prec ... \prec x_n^*$, with, for $i \in \llbracket1,~ n\rrbracket,~ supp(x_i^*) \subseteq \llbracket p_i,~ q_i\rrbracket$, where $q_i < p_{i+1}$ for $i \in \llbracket 1,~ n-1\rrbracket$. Fix now $\varepsilon > 0$.

\noindent Since $(e_i)_{i=0}^{\infty}$ is a monotone basis: $$\exists~x_i \in J_p,~\left\{\begin{array}{rl}x_i^*(x_i)&\geq~ ||x_i^*||_{J_p^*}^q - \varepsilon\\||x_i||_{J_p}&\leq~ 2||x_i^*||_{J_p^*}^{q-1} \\supp(x_i)&\subseteq~ \llbracket p_i,~ q_i\rrbracket\end{array}\right.$$

\noindent Thanks to Lemma \ref{L1}~: $\exists~ C > 0,~||\sum\limits_{i=1}^{n}x_i||_{J_p}^p \leq C\sum\limits_{i=1}^{n}||x_i||_{J_p}^p$ (where $C = 2^p + 1$).

\noindent Since $||.||_{J_p^*}$ is the dual norm of $||.||_{J_p}$, we have that $$||\sum\limits_{i=1}^{n}x^*_i||_{J_p^*} \geq \left(\sum\limits_{i=1}^{n}x_i^*\right)\left(\sum\limits_{i=1}^{n}x_i\right)
\left(||\sum\limits_{i=1}^{n}x_i||_{J_p}\right)^{-1}$$
and
$$||\sum\limits_{i=1}^{n}x_i^*||_{J_p^*} \geq \left(\sum\limits_{i=1}^{n}x_i^*(x_i)\right)\left(||\sum\limits_{i=1}^{n}x_i||_{J_p}\right)^{-1} \geq \left(\sum\limits_{i=1}^{n}||x_i^*||_{J_p^*}^q - \varepsilon\right)\left(C^{\frac{1}{p}}
\left(\sum\limits_{i=1}^{n}||x_i||^p_{J_p}\right)^{\frac{1}{p}}\right)^{-1}.$$

\noindent Moreover, $\left(\sum\limits_{i=1}^{n}||x_i||^p_{J_p}\right)^{\frac{1}{p}} \leq 2\left(\sum\limits_{i=1}^{n}||x_i^*||^{p(q-1)}_{J_p^*}\right)^{\frac{1}{p}} = 2\left(\sum\limits_{i=1}^{n}||x_i^*||^q_{J_p^*}\right)^{\frac{1}{p}}$.

\noindent Letting $\varepsilon$ tend to $0$, we obtain~: $$||\sum\limits_{i=1}^{n}x_i^*||_{J_p^*} \geq \dfrac{1}{2C^\frac{1}{p}}\left(\sum\limits_{i=1}^{n}||x_i^*||^q_{J_p^*}\right)^{1 - \frac{1}{p}} = \dfrac{1}{2C^\frac{1}{p}}\left(\sum\limits_{i=1}^{n}||x_i^*||^q_{J_p^*}\right)^{\frac{1}{q}}.$$
So, we have established inequality (\ref{1}) with $c = \dfrac{1}{2^qC^{\frac{q}{p}}} = \dfrac{1}{2^qC^{q-1}}$.

\noindent It follows easily that
$$||x^{*}||_{J_p^{*}} \leq |x^{*}|_{J_p^{*}} \leq c^{-1/q}||x^{*}||_{J_p^{*}}.$$

\noindent Moreover, $|.|_{J_p^{*}}$ is the dual norm of an equivalent norm on $J_p$. Indeed, it is clear that $|.|_{J_p^{*}}$ is $\sigma(J_p^{*},J_p)$ lower semi-continuous.

\noindent Finally, it follows clearly from the definition of $|.|_{J_p^*}$ that for all  $x^*,y^*$ in $J_p^*$ such that $x^{*} \prec y^{*}$, we have that
$$|x^{*} + y^{*}|^q_{J_p^{*}} \geq |x^{*}|^q_{J_p^{*}} + |y^{*}|^q_{J_p^{*}}.$$

\end{proof}

\begin{Corollary}\label{C1}
The dual norm $|.|_{J_p^{**}}$ of $|.|_{J_p^{*}}$ satisfies the following property.

\noindent For $x \in J_p$ with a finite support and $y \in J_p^{**}$ (not necessarily with finite support) such that $x \prec y$, we have $$|x + y|_{J_p^{**}}^p \leq |x|_{J_p^{**}}^p + |y|_{J_p^{**}}^p.$$
\end{Corollary}

\begin{proof}
Let $x \in J_p$ which has a finite support and $y \in J_p^{**}$ such that $x \prec y$, with $supp(x) \subset~ \llbracket m,~ n\rrbracket$, $supp(y) \subset \llbracket m',~ +\infty)$ and $n < m'$. Fix $\varepsilon > 0$.

\noindent There exists $z^{*} \in J_p^{*}$ such that $$|z^{*}|_{J_p^{*}} = |x + y|^{p-1}_{J_p^{**}} ~~~~ \textnormal{ and } ~~~~z^{*}(x + y) \geq |x + y|^p_{J_p^{**}} - \varepsilon.$$

\noindent Moreover, we can write $z^{*} = x^{*} + y^{*}$, with $x^{*} \prec y^{*}$, $z^{*}(x) = x^{*}(x)$ and $z^{*}(y) = y^{*}(y)$.

\noindent We deduce that $|x + y|^p_{J_p^{**}} \leq x^{*}(x) + y^{*}(y) + \varepsilon$. Then H\"older's inequality and Proposition \ref{P0} yield
$$|x + y|^p_{J_p^{**}} \leq (|x^{*}|^q_{J_p^{*}} + |y^{*}|^q_{J_p^{*}})^{\frac{1}{q}}(|x|^p_{J_p^{**}} + |y|^p_{J_p^{**}})^{\frac{1}{p}} + \varepsilon \le (|x^{*} + y^{*}|_{J_p^{*}})(|x|^p_{J_p^{**}} + |y|^p_{J_p^{**}})^{\frac{1}{p}} + \varepsilon.$$
Since $|z^{*}|_{J_p^{*}} = |x + y|^{p-1}_{J_p^{**}}$, we get
$$|x + y|^p_{J_p^{**}} \leq (|x + y|^{p-1}_{J_p^{**}})(|x|^p_{J_p^{**}} + |y|^p_{J_p^{**}})^{\frac{1}{p}} + \varepsilon.$$
This finishes our proof.
\end{proof}

We now turn to the study of the coarse Lipschitz embeddings between James spaces. Let us first recall some notation.
\begin{Definition} Let $(M,d)$ and $(N,\delta)$ be two metric
spaces and $f:M\to N$ be a mapping. If $(M,d)$ is unbounded, we define
$$\forall s>0,\ \ Lip_s(f)=\sup\Big\{\frac{\delta((f(x),f(y))}{d(x,y)},\ d(x,y)\ge s\Big\}\ \
{\rm and}\ \ Lip_\infty(f)=\inf_{s>0}Lip_s(f).$$
\end{Definition}
Note that $f$ is coarse Lipschitz if and only if $Lip_\infty(f)<\infty$.

\medskip We also recall a classical definition.

\begin{Definition}
Given a metric space $X$, two points $x,~y \in X$, and $\delta > 0$, the approximate metric midpoint set between $x$ and $y$ with error $\delta$ is the set~: $$Mid(x,~y,~\delta) = \left\{z \in X ~:~ \textnormal{max}\{ d(x,~z),~d(y,~z)\} \leq (1 + \delta)\dfrac{d(x,~y)}{2}\right\}$$
\end{Definition}

The use of approximate metric midpoints in the study of nonlinear geometry is due to Enflo in an unpublished paper and has been used elsewhere, e.g. \cite{Bo}, \cite{G} and \cite{Jo}.
The next proposition and its proof can be found for instance in \cite{KR} and \cite{L1}.

\begin{Proposition}\label{P2} Let $X$ be a normed space and suppose $M$ is a metric space. Let $f~: X \rightarrow M$ be a coarse Lipschitz map. If $Lip_{\infty}(f) > 0$, then for any $t,~\varepsilon > 0$ and any $0 < \delta < 1$, there exist $x,~y \in X$ with $||x - y|| > t$ and $$f\left(Mid(x,~y,~\delta)\right) \subset Mid\left(f(x),~f(y),~(1 + \varepsilon)\delta\right).$$
\end{Proposition}

Let us now recall the definition of the metric graphs introduced in \cite{KR} that will be crucial in our proofs.

\begin{Notation}
Let $\mathbb{M}$ be an infinite subset of $\mathbb{N}$ and $k \in \mathbb{N}$. We denote $$G_k(\mathbb{M}) = \{\overline{n} = (n_1, ..., n_k),~n_i \in \mathbb{M}\hspace{0.5cm} n_1 < ... < n_k\}.$$
Then we equip $G_k(\mathbb{M})$ with the distance $d(\overline{n},\overline{m}) = |\{j,~n_j \neq m_j\}|$.
\end{Notation}

We end these preliminaries by recalling Ramsey's theorem and one of its immediate  corollaries (see \cite{GR} for instance).
\begin{Theorem}
Let $k,~r \in \mathbb{N}$ and $f~: G_k(\mathbb{N}) \rightarrow \{1,...,r\}$ be any map. Then there exists an infinite subset $\mathbb{M}$ of $\mathbb{N}$ and $i \in \{1,...,r\}$ such that, for every $\overline{n} \in G_k(\mathbb{M})$, $f(\overline{n}) = i$.
\end{Theorem}

\begin{Corollary}
Let $(K,d)$ be a compact metric space, $k \in \mathbb{N}$ and $f~: G_k(\mathbb{N}) \rightarrow K$. Then for every $\epsilon > 0$, there exist an infinite subset $\mathbb{M}$ of $\mathbb{N}$ such that for every $\overline{n},\overline{m} \in G_k(\mathbb{M}),~d(f(\overline{n}),f(\overline{m})) < \epsilon$.
\end{Corollary}

\section{The main results}

Our first Lemma gives a description of approximate metric midpoints in $J_p$ that is analogous to situation in $\ell_p$ (see \cite{KR} or \cite{L1}). However, we need to use both the original and our new norm on $J_p$.

\begin{Lemma} Let $1 < p < \infty$. We denote $E_N$ the closed linear span of \{$e_i$, $i > N$\}. Let now $x,~ y \in J_p,~ \delta \in (0,~ 1)$, $u = \dfrac{x + y}{2}$ and $v = \dfrac{x - y}{2}$. Then

\noindent (i) There exists $N \in \mathbb{N}$ such that:
\begin{center} $u + \delta^{\frac{1}{p}}|v|_{J_p}B_{(E_N, |.|_{J_p})} \subset Mid_{|.|_{J_p}}(x,~ y,~ \delta)$. \end{center}

\noindent (ii) There is a compact subset $K$ of $J_p$ such that:
\begin{center} $Mid_{||.||_{J_p}}(x,~ y,~ \delta) \subset K + 2\delta^{\frac{1}{p}}||v||_{J_p}B_{(J_p, ||.||_{J_p})}$. \end{center}
\end{Lemma}

\begin{proof}
Fix $\lambda > 0$.\\
Let $N \in \mathbb{N}$ such that $|v - v_N|_{J_p} \leq \lambda|v|_{J_p}$ and $|v_N|_{J_p}^p \geq (1+\lambda^p)^{-1}|v|_{J_p}^p$, where $v_N = \sum_{i=1}^N v(i)e_i$.

\smallskip (i) Let now $z \in E_N$ so that $|z|_{J_p}^p \leq \delta |v|_{J_p}^p$. Then
$$|x - (u + z)|_{J_p}^p = |v - z|_{J_p}^p = |v - v_N + v_N - z|_{J_p}^p$$

\noindent It follows from the Corollary \ref{C1} that:
$$|x - (u + z)|_{J_p}^p \leq |v - v_N - z|_{J_p}^p + |v_N|_{J_p}^p\leq (|v - v_N|_{J_p} + |z|_{J_p})^p + |v|_{J_p}^p$$
Therefore
$$|x - (u + z)|_{J_p}^p \leq \big((1 + \delta^{1/p})^p+1\big)|v|_{J_p}^p \le (1+\delta)^p|v|_{J_p}^p,$$

\noindent if $\lambda$ was chosen initially small enough.

\noindent We argue similarly to show that $|y - (u + z)|_{J_p} = |v + z|_{J_p}\le (1+\delta)|v|_{J_p}$ and deduce that $u + z \in Mid(x,~ y,~ \delta)$.

\smallskip (ii) Fix $\nu>0$ and choose $N\in \mathbb N$ such that $\|v_N\|_{J_p}^p\ge (1-\nu^p)\|v\|^p$. We assume now that $u + z \in Mid_{||.||_{J_p}}(x,~ y,~ \delta)$ and write $z = z' + z''$ with $z' \in F_N = sp\{e_i, i \leq N\}$ and $z'' \in E_N$.

\noindent Since $||v - z||_{J_p}~,~ ||v + z||_{J_p} \leq (1 + \delta)||v||_{J_p}$, we get, by convexity, that $$||z'||_{J_p} \leq ||z||_{J_p} \leq (1 + \delta)||v||_{J_p}.$$

\noindent Therefore, $u + z'$ belongs to the compact set $K = u + (1 + \delta)||v||_{J_p}B_{({F_N}, ||.||_{J_p})}$.

\noindent Moreover, for any $(m,n) \in (\mathbb{N}^{*})^2$, with $m > n$:
$$\textnormal{max}\{|v(n) - v(m)|^p,~ |z(n) - z(m)|^p\} \leq \dfrac{1}{2}(|(v(n) - z(n)) - (v(m) - z(m))|^p$$

\hspace{7cm} $ + |(v(n) + z(n)) - (v(m) + z(m))|^p).$

\noindent Therefore
$$(1 - \nu^p)||v||_{J_p}^p + ||z''||_{J_p}^p \leq ||v_N||_{J_p}^p + ||z''||_{J_p}^p \leq \dfrac{1}{2}(||v - z||_{J_p}^p + ||v + z||_{J_p}^p)\leq (1 + \delta)^p||v||_{J_p}^p.$$
Then, if $\nu$ was chosen small enough, we get
$$||z''||_{J_p}^p \leq [(1 + \delta)^p - (1 - \nu^p)]||v||_{J_p}^p \leq 2^p\delta||v||_{J_p}^p.$$
\end{proof}

\begin{Proposition}\label{P3}
Let $1 < p < q < \infty$ and $f: (J_q,~|.|_{J_q}) \rightarrow (J_p,~||.||_{J_p})$ be a coarse Lipschitz embedding. Then, for any $t > 0$ and for any $\varepsilon > 0$, there exist $u \in J_q$, $\theta > t$, $N \in \mathbb{N}$ and $K$ a compact subset of $J_p$ such that
\begin{center} $f(u + \theta B_{(E_N, |.|_{J_q})}) \subset K + \varepsilon \theta B_{(J_p, ||.||_{J_p})}$. \end{center}
\end{Proposition}

\begin{proof} If $Lip_{\infty}(f) = 0$, the conclusion is clear. So we assume that $Lip_{\infty}(f) > 0$.

\noindent We choose a small $\delta > 0$ (to be detailed later). Then we choose $s$ large enough so that $Lip_{s}(f) \leq 2 Lip_{\infty}(f)$.\\ Then, by Proposition \ref{P2},
\begin{center} $\exists~x,~y \in J_q,~ |x - y|_{J_q} \geq s$ ~and~ $f(Mid_{|.|_{J_q}}(x,~y,~\delta)) \subset Mid_{||.||_{J_p}}(f(x),~f(y),~2\delta)$. \end{center}

\noindent Denote $u = \dfrac{x + y}{2}$, $v = \dfrac{x - y}{2}$ and $\theta = \delta^{\frac{1}{q}} |v|_{J_q}$. By Lemma 3.1, there exists $N \in \mathbb{N}$ such that $u + \theta B_{(E_N, |.|_{J_q})} \subset Mid_{|.|_{J_q}}(x,~y,~\delta)$ and there exists a compact subset $K$ of $J_p$ so that $Mid_{||.||_{J_p}}(f(x),~f(y),~2\delta) \subset K + 2(2\delta)^{\frac{1}{p}}||f(x) - f(y)||_{J_p}B_{(J_p, ||.||_{J_p})}$. But~:
$$2(2\delta)^{\frac{1}{p}}||f(x) - f(y)||_{J_p} \leq 4 Lip_{\infty}(f)(2\delta)^{\frac{1}{p}}|x - y|_{J_q}$$

\hspace{6.1cm} $\leq 8 Lip_{\infty}(f)2^{\frac{1}{p}}\delta^{\frac{1}{p} - \frac{1}{q}}\theta \leq \varepsilon\theta$,

\noindent if $\delta$ was chosen initially small enough.

\noindent Then an appropriate choice of a large $s$ will ensure that $\theta \geq \dfrac{1}{2}\delta^{\frac{1}{q}}s > t$. This finishes the proof.
\end{proof}

\begin{Corollary}\label{C2}
Let $1 < p < q < \infty$.

\noindent Then $J_q$ does not coarse Lipschitz embed into $J_p$.
\end{Corollary}

\begin{proof} We proceed by contradiction and suppose that there exists a coarse Lipschitz embedding $f~: (J_q,~|.|_{J_q}) \rightarrow (J_p,~||.||_{J_p})$.

\noindent With the notation of the previous proposition, we can find a sequence $(u_n)_{n=1}^{\infty}$ in $u + \theta B_{(E_N, |.|_{J_q})}$, such that $|u_n - u_m|_{J_q} \geq \theta$ for $n \neq m$. Then $f(u_n) = k_n + \varepsilon\theta v_n$, with $k_n \in K$ et $v_n \in B_{(J_p, ||.||_{J_p})}$. Since $K$ is compact, by extracting a subsequence, we may assume that $||f(u_n) - f(u_m)||_{J_p} \leq 3\varepsilon\theta$.

\noindent Since $\varepsilon$ can be chosen arbitrarily small and $\theta$ arbitrarily large, this yields a contradiction.
\end{proof}

In order to treat the coarse Lipschitz embeddability in the other direction, we shall use the Kalton-Randrianarivony graphs and some special subsets of them that we introduce now.

\begin{Definition}
Let $\overline{n}$, $\overline{m} \in G_k(\mathbb{M})$ (where $\mathbb{M}$ is an infinite subset of $\mathbb{N}$).

\noindent We say that $(\overline{n},~ \overline{m}) \in I_k(\mathbb{M})$ if $n_1 < m_1 < n_2 < m_2 < ... < n_k < m_k$.
\end{Definition}

\begin{Proposition}\label{P4}
Let $\varepsilon > 0$ and $f~:~ G_k(\mathbb{N}) \rightarrow (J_p^{**},|\ |_{J_p^{**}})$ be a Lipschitz map.

\noindent Then, for any infinite subset $\mathbb{M}$ of $\mathbb{N}$, there exists $(\overline{n},~\overline{m}) \in I_k(\mathbb{M})$ such that $$|f(\overline{n}) - f(\overline{m})|_{J_p^{**}} \leq 2 Lip(f)k^{\frac{1}{p}} + \varepsilon$$
\end{Proposition}

\begin{proof}
We shall prove this statement by induction on $k \in \mathbb{N}$

\noindent The proposition is clearly true for $k = 1$.

\noindent Assume now that it is true for $k-1 \geq 1$.

\noindent Let $f~:~G_k(\mathbb{M}) \rightarrow J_p^{**}$ be a Lipschitz map and $\varepsilon > 0$.

\noindent Let $\eta > 0$ (small enough~: to be detailed later).

\noindent By a diagonal extraction process and thanks to the weak$^*$-compactness, we can find an infinite subset $\mathbb{M}_1$ of $\mathbb{N}$ such that
\begin{equation}\label{cv}  \forall~ \overline{n} \in G_{k-1}(\mathbb{N}),~w^{*}- \underset{k \in \mathbb{M}_1}{lim}~f(\overline{n},k) = g(\overline{n}) \in J_p^{**}
\end{equation}

\noindent Then $Lip(g) \leq Lip(f)$, by weak$^*$-lower semicontinuity of $|.|_{J_p^{**}}$.

\noindent Denote $g(\overline{n}) = v(\overline{n}) + c_{\overline{n}}\textrm{1\kern-0.25emI}$ (where $v(\overline{n}) \in J_p$, $c_{\overline{n}} \in \mathbb{R}$ and $\textrm{1\kern-0.25emI}$ is the constant sequence $(1,1,1,...)$).

\noindent By Ramsey's theorem, there exists an infinite subset $\mathbb{M}_2$ of $\mathbb{M}_1$ such that
\begin{equation}\label{2}  \forall~ \overline{n},\overline{m} \in G_{k-1}(\mathbb{M}_2)\ \ \  |c_{\overline{n}} - c_{\overline{m}}| =|(v(\overline{n}) - v(\overline{m})) - (g(\overline{n}) - g(\overline{m}))|_{J_p^{**}}< \eta.
\end{equation}

\noindent For $\overline{n},\overline{m} \in G_{k-1}(\mathbb{M}_2)$ and $t,l\in \mathbb{M}_2$, set
$$u_{\overline{n},\overline{m},t,l} = f(\overline{n},t) - g(\overline{n}) + g(\overline{m}) - f(\overline{m},l).$$
Using (\ref{cv}) and Corollary \ref{C1} we deduce that there exists $ l_0 \in \mathbb{N}$ such that for all $t,l \in \mathbb{M}_2 \cap [l_0,~+\infty[$:

$$|v(\overline{n})-v(\overline{m})+u_{\overline{n},\overline{m},t,l}|^p\le |v(\overline{n})-v(\overline{m})|^p+|u_{\overline{n},\overline{m},t,l}|^p+\eta.$$
Then it follows from (\ref{2}) that for all $t,l \in \mathbb{M}_2 \cap [l_0,~+\infty[$:
$$|f(\overline{n},t) - f(\overline{m},l)|_{J_p^{**}}^p \leq |u_{\overline{n},\overline{m},t,l}|_{J_p^{**}}^p + (|v(\overline{n}) - v(\overline{m})|_{J_p} + \eta)^p + \eta.$$
Moreover~: $f(\overline{n},t) - g(\overline{n}) = w^*-\underset{i}{lim}(f(\overline{n},t)-f(\overline{n},i))$.

\noindent Therefore, by weak$^*$-lower semicontinuity of $|.|_{J_p^{**}}$~: $|f(\overline{n},t) - g(\overline{n})|_{J_p^{**}} \leq Lip(f)$.

\noindent Likewise~: $|f(\overline{m},l) - g(\overline{m})|_{J_p^{**}} \leq Lip(f)$.

\noindent Then, we deduce the following inequality~: $|u_{\overline{n},\overline{m},t,l}|_{J_p^{**}}^p \leq 2^p Lip(f)^p$.

\noindent On the other hand, it follows from our induction hypothesis that:
$$\exists~ (\overline{n},~\overline{m}) \in~(I_{k-1}(\mathbb{M}_2))^2,~ |g(\overline{n}) - g(\overline{m})|_{J_p^{**}} \leq 2Lip(f)(k-1)^{\frac{1}{p}} + \eta.$$

\noindent Then, for $t,~l \in \mathbb{M}_2 \cap [l_0,~+\infty[$ such that $m_{k-1} < t < l$, we have $((\overline{n},t),(\overline{m},l)) \in I_k(\mathbb{M}_2)$, and 
$$|f(\overline{n},t) - f(\overline{m},l)|_{J_p^{**}}^p \leq 2^pLip(f)^p + (2Lip(f)(k-1)^{\frac{1}{p}} + 2\eta)^p + \eta.$$

\noindent So~: $$|f(\overline{n},t) - f(\overline{m},l)|_{J_p^{**}}^p \leq 2^pLip(f)^pk + \varphi(\eta), \textnormal{~with~}\varphi(\eta) \underset{\eta \to 0}{\rightarrow} 0.$$

\noindent Thus, if $\eta$ was chosen small enough~: $$|f(\overline{n},t) - f(\overline{m},l)|_{J_p^{**}} \leq 2 Lip(f)k^{1/p} + \varepsilon.$$

\noindent This finishes our inductive proof.
\end{proof}

\begin{Corollary}\label{C3}
Let $1 < q < p < \infty$.

\noindent Then $J_q$ does not coarse Lipschitz embed into $J_p$.
\end{Corollary}

\begin{proof}
Suppose that $g~: J_q \rightarrow J_p$ is a map such that there exist $\theta$, $A$ and $B$ real positive numbers such that~:
$$\forall x,y \in J_q\ \  ||x - y||_{J_q} \geq \theta \Rightarrow A||x - y||_{J_q} \leq |g(x) - g(y)|_{J_p} \leq B||x - y||_{J_q}.$$

\noindent Let us rescale by defining  $f(v)={A\theta}^{-1}g(\theta v)$, for $v \in J_q$. We have that there exists $C\ge 1$ such that

\begin{equation}\label{3}
\forall x,y \in J_q, ||x - y||_{J_q} \geq 1 \Rightarrow ||x - y||_{J_q} \leq |f(x) - f(y)|_{J_p} \leq C||x - y||_{J_q}.
\end{equation}

\noindent We still denote $(e_n)_{n=1}^{\infty}$ the canonical basis of $J_q$.

\noindent Consider the map $\varphi~: G_{k}(\mathbb{N}) \rightarrow (J_q,\|\ \|_{J_q})$ defined by $\varphi(\overline{n}) = e_{n_1} + ... + e_{n_k}$.\\
Note that $Lip(\varphi) \leq 4$. Since $\|\varphi(\overline{n}) - \varphi(\overline{m})\|_{J_q} \geq 1$ whenever $\overline{n} \neq \overline{m}$, we have that
$Lip(f \circ \varphi) \leq 4C$, where $f \circ \varphi$ is considered as a map from $G_{k}(\mathbb{N})$ to $(J_p,|\ |_{J_p})$. It then follows from Proposition \ref{P4} that there exist $(\overline{n},\overline{m}) \in I_{k}(\mathbb{N})$  such that:
$$|(f \circ \varphi)(\overline{n}) - (f \circ \varphi)(\overline{m})|_{J_p} \leq 9Ck^{\frac{1}{p}}.$$

\noindent On the other hand, since $(\overline{n},\overline{m}) \in I_k(\mathbb{N})$, we have that 
$$||\varphi(\overline{n}) - \varphi(\overline{m})||_{J_q} \geq 2(2k-1)^{\frac{1}{q}}\ge 2k^{\frac{1}{q}}.$$

\noindent This is in contradiction with (\ref{3}), for $k$ large enough.

\noindent Therefore, there is no coarse Lipschitz embedding from $J_q$ into $J_p$.
\end{proof}

\begin{Corollary}
Let $1 < p < q < \infty$, and $r > 1$ such that $r \notin \{p,~q\}$.

\noindent Then $J_r$ does not coarse Lipschitz embed into $J_p \oplus J_q$.
\end{Corollary}

\begin{proof} When $r > q$, the argument is based on a midpoint technique like in the proof of Corollary \ref{C2}.

\smallskip If $r < p$, we mimic the proof of Corollary \ref{C3}.

\smallskip So we assume, as we may,  that $1 < p < r < q < \infty$ and $f~: J_r \rightarrow J_p \oplus_{\infty} J_q$ is a map such that there exists $C\ge 1$ such that
\begin{equation}\label{4}
\forall~ x,~y \in J_r\ \ |x - y|_{J_r} \geq 1 \Rightarrow |x - y|_{J_r} \leq ||f(x) - f(y)|| \leq C|x - y|_{J_r}.
\end{equation}

\noindent We follow the proof in \cite{KR} and write $f = (g,h)$. We still denote $(e_n)_{n=1}^{\infty}$ the canonical basis of $J_r$. We fix $k \in \mathbb{N}$ and $\varepsilon > 0$. We recall that 
$$\exists~ \gamma > 0,~\forall~ x \in J_r,~ \gamma||x||_{J_r} \leq |x|_{J_r} \leq ||x||_{J_r}.$$

\noindent We start by applying the midpoint technique to the coarse Lipschitz map $g$ and deduce from Proposition \ref{P3} that there exist $\theta > \gamma^{-1}(2k)^{1/r}$, $u \in J_r$, $N \in \mathbb{N}$ and $K$ a compact subset of $J_p$ such that~:

\begin{equation}\label{5}
g(u + \theta B_{(E_N,|.|_{J_r})}) \subset K + \varepsilon\theta B_{(J_p,||.||_{J_p})}.
\end{equation}

\noindent Let $\mathbb{M} = \{n \in \mathbb{N},~n > N\}$
and $\varphi~: G_k(\mathbb{M}) \mapsto J_r$ be defined as follows 
$$\forall~\overline{n} = (n_1,...,n_k) \in G_k(\mathbb{M}),~\varphi(\overline{n}) = u + \theta(2k)^{-\frac{1}{r}}(e_{n_1} + ... + e_{n_k}).$$

\noindent Then $\varphi(\overline{n}) \in u + \theta B_{(E_N,|.|_{J_r})}$ for all $\overline{n} \in G_k(\mathbb{M})$.

\noindent And, from (\ref{5}) we deduce that $(g \circ \varphi) (G_k(\mathbb{M})) \subset K + \varepsilon\theta B_{(J_p,||.||_{J_p})}$. Thus, by Ramsey's theorem, there is an infinite subset $\mathbb{M}'$ of $\mathbb{M}$ such that 

\begin{equation}\label{6}
\textnormal{diam}_{\|\ \|_{J_p}}(g \circ \varphi)(G_k(\mathbb{M}')) \leq 3\varepsilon\theta.
\end{equation}

\noindent Since for $\overline{n} \neq \overline{m}$, we have 
$$|\varphi(\overline{n}) - \varphi(\overline{m})|_{J_r}\ge \gamma||\varphi(\overline{n}) - \varphi(\overline{m})||_{J_r} \geq \gamma\theta (2k)^{-\frac{1}{r}} > {1},$$
it follows from (\ref{4}) that
$$\forall \overline{n}, \overline{m}\in G_k(\mathbb{M})\ \ ||h \circ \varphi(\overline{n}) - h \circ \varphi(\overline{m})||_{J_q} \leq C|\varphi(\overline{n}) - \varphi(\overline{m})|_{J_r}\leq C||\varphi(\overline{n}) - \varphi(\overline{m})||_{J_r}.$$

\noindent Note now that for all $\overline{n}, \overline{m}\in G_k(\mathbb{N})$, 
$$||(e_{n_1} + ... + e_{n_k}) - (e_{m_1} + ... + e_{m_k})||_{J_q} \leq \displaystyle\sum\limits_{n_i \neq m_i}^{}||e_{n_i} - e_{m_i}||_{J_q} \leq 2d(\overline{n},\overline{m}).$$
Since moreover $|\ |_{J_q}\le\|\ \|_{J_q}$, $Lip(h \circ \varphi) \leq 2C\theta (2k)^{-\frac{1}{r}}$, when $h \circ \varphi$ is considered as a map from $G_k(\mathbb{M}')$ to $(J_q,|\ |_{J_q})$. Thus, we can apply Proposition \ref{P4} to obtain:
$$\exists~ (\overline{n},~\overline{m}) \in I_k(\mathbb{M}'),~ |h \circ \varphi(\overline{n}) - h \circ \varphi(\overline{m})|_{J_q} \leq 5C\theta(2k)^{-\frac{1}{r}}k^{\frac{1}{q}}.$$

\noindent Then, if $k$ was chosen large enough, we have: 
$$\exists~ (\overline{n},~\overline{m}) \in I_k(\mathbb{M}'),~ |h \circ \varphi(\overline{n}) - h \circ \varphi(\overline{m})|_{J_q} \leq \varepsilon\theta.$$

\noindent This, combined with (\ref{6}) implies that 
$$\exists~ (\overline{n},~\overline{m}) \in I_k(\mathbb{M}')\ \ \|f\circ\varphi(\overline{n}) - f\circ\varphi(\overline{m})\| \le3\varepsilon\theta.$$
But, 
$$\forall (\overline{n},~\overline{m}) \in I_k(\mathbb{M}')\ \ |\varphi(\overline{n}) - \varphi(\overline{m})|_{J_r}\ge \gamma\|\varphi(\overline{n}) - \varphi(\overline{m})\|_{J_r}\ge \gamma\theta.$$
If $\varepsilon$ was initially chosen such that $\varepsilon < \frac{\gamma}{3}$, this yields a contradiction with (\ref{4}), which concludes our proof.
\end{proof}

\noindent{\bf Remark.} This result can be easily extended as follows. Assume $r \in (1,\infty) \setminus \{p_1,~...,~p_n\}$ where $1 < p_1 < p_2 < ... <p_n < \infty$, then $J_r$ does not coarse Lipschitz embed into $J_{p_1} \oplus ... \oplus J_{p_n}$.

\begin{center}\textsc{Aknowledgment}\end{center}

The author is grateful to Gilles Lancien for sharing his knowledge and experience, and for helpful discussions that contributed to shape this paper.

\end{document}